\documentclass[aap,preprint]{imsart}

\usepackage{amsthm,amsmath,amssymb, verbatim, comment, color, amsfonts,amsxtra}
\usepackage{centernot}
\RequirePackage[colorlinks,citecolor=blue,urlcolor=blue]{hyperref}
\RequirePackage[OT1]{fontenc}

\startlocaldefs

\newtheorem{theorem}{Theorem}

\newtheorem{corollary}[theorem]{Corollary}
\newtheorem{definition}[theorem]{Definition}
\newtheorem{lemma}[theorem]{Lemma}
\newtheorem{proposition}[theorem]{Proposition}

\newtheorem{claim}[theorem]{Claim}
\theoremstyle{remark}
\newtheorem{remark}[theorem]{Remark}
\newtheorem{example}[theorem]{Example}

 \newcommand{\eps}{\varepsilon}

 \renewcommand{\phi}{\varphi}
 
 \newcommand{\OB}{Ob{\l}{\'o}j} 
\newcommand{\HL}{Henry-Labord{\`e}re}

\renewcommand{\Cap}{\bigcap}
\newcommand{\E}{\mathbb{E}}
\renewcommand{\P}{\mathbb{P}}
\newcommand{\N}{\mathbb{N}}

\newcommand{\R}{\mathbb{R}}

\newcommand{\be}{\begin{equation}}
\newcommand{\ee}{\end{equation}}

\definecolor{darkviolet}{rgb}{0.58, 0.0, 0.83}

\DeclareMathOperator{\supp}{supp}

\numberwithin{equation}{section}
\numberwithin{theorem}{section}

\usepackage{enumerate}

\setattribute{journal}{name}{}

\endlocaldefs

\begin{document}

\begin{frontmatter}

\title{Optimal Brownian stopping when the source and target  
are radially symmetric distributions
} 

\runtitle{Optimal Brownian stopping}

\thankstext{}{The  first two  
authors are partially supported by  the 
Natural Sciences and Engineering Research Council of Canada (NSERC). Y.-H. Kim was also supported by an Alfred P. Sloan Research Fellowship. T. Lim was partly supported by a doctoral graduate fellowship from the University of British Columbia, by the Austrian Science Foundation (FWF) through grant Y782, and by the European Research Council under the European Union's Seventh Framework Programme (FP7/2007-2013) / ERC grant agreement no. 335421. In addition, T. Lim gratefully acknowledges support from ShanghaiTech University. Part of this research was done while the authors  were visiting the Fields institute in Toronto during the thematic program on ``Calculus of Variations'' in Fall 2014. We are thankful for the hospitality and the great research environment that the institute provided. 
}

\begin{aug}
\author{\fnms{Nassif} \snm{Ghoussoub}\thanksref{m1}
\ead[label=e1]{nassif@math.ubc.ca}}
\author{\fnms{Young-Heon} \snm{Kim}\thanksref{m1}
\ead[label=e2]{yhkim@math.ubc.ca,}}
\and
\author{\fnms{Tongseok} \snm{Lim}\thanksref{m2}
\ead[label=e3]{TLIM@shanghaitech.edu.cn}}
%\ead[label=u1,url]{http://www.foo.com}}

%\thankstext{t1}{Some comment}
%\thankstext{t2}{First supporter of the project}
%\thankstext{t3}{Second supporter of the project}
\runauthor{N. Ghoussoub, Y-H Kim and T. Lim}

\affiliation{The University of British Columbia\thanksmark{m1} and ShanghaiTech University \thanksmark{m2}}

\address{Department of Mathematics\\University of British Columbia\\ Vancouver, V6T 1Z2 Canada\\
\printead{e1}\\
\phantom{E-mail:\ }\printead*{e2}}

\address{Institute of Mathematical Sciences\\ShanghaiTech University\\
No. 393, Huaxia Middle Road, Pudong New Area, Shanghai 201210\\
\printead{e3}}
%\printead{u1}}
\end{aug}

\begin{abstract}
Given two probability measures $\mu, \nu$ on $\R^d$, {\em in subharmonic order}, we describe optimal stopping times $\tau$ that  maximize/minimize the cost functional $\E |B_0 - B_\tau|^{\alpha}$, $\alpha > 0$, where $(B_t)_t$ is Brownian motion with initial law $\mu$ and with final distribution --once stopped at $\tau$--  equal to $\nu$. Under the assumption of radial symmetry  on $\mu$ and $\nu$, we show that in dimension $d \geq 3$ and $\alpha \neq 2$, there exists a unique optimal solution given by a non-randomized stopping time characterized as the hitting time to a suitably symmetric {\em barrier}.  We also relate  this problem to the optimal transportation problem for {\em subharmonic martingales}, and establish a duality result. This paper is an expanded version of a previously posted but not published work by the authors \cite{GKL2}.

\end{abstract}


\tableofcontents
\end{frontmatter}

%\begin{abstract}
%\end{abstract}

%\begin{keyword}[class=MSC]
%\kwd[Primary ]{}
%\kwd{}
%\kwd[; secondary ]{}
%\end{keyword}

%\begin{keyword}
%\kwd{}
%\kwd{}
%\end{keyword}

%\end{frontmatter}

\section{Introduction}\label{intro}  

Let  $\mu$ and $\nu$ be two probability measures on $\R^d$, $d \geq 2$ with finite first moment, and let $(B_t)_t$ denote the  Brownian motion with initial law $\mu$.  
We consider the following --possibly empty-- set of stopping times, with respect to the Brownian filtration:
\begin{align*} 
{\cal T}(\mu, \nu) = \{\tau \,|\, \text{$\tau$ is a stopping time, } B_0 \sim \mu, B_\tau \sim \nu, \text{ and } \E[\tau] < \infty \}, 
\end{align*} 
where here and in the sequel, the notation $X\sim \lambda$ means that the law of the random variable $X$ is the probability measure $\lambda$.  

For a given cost function $c: \R^d \times \R^d \to \R$, and assuming ${\cal T}(\mu, \nu)$ is non-empty, we shall consider  the following  optimization problem
\begin{align}\label{opt}
 \text{ Maximize / Minimize }\ \E\, [c(B_0, B_\tau)] \quad \text{over} \quad \tau \in {\cal T}(\mu, \nu).
\end{align}

Recall that a {\em stopping time} on a filtered probability space $(\Omega, \mathcal F, (\mathcal F_t)_t,\mathbb P)$ is a random variable $\tau :\Omega \to [0, +\infty]$ such that $\{\tau \leq t\}\in \mathcal F_t$ for every $t\geq 0$. A {\em randomized stopping time} is a probability measure $\tau$ on $\Omega \times [0, +\infty]$ such that for each $u \in \R_+$, 
the random time $\rho_u (\omega) :=\inf\{ t\ge 0: \tau_{\omega} ([0,t]) \ge u\}
$  is a stopping time, where  $(\tau_{\omega})_\omega$ is a disintegration of $\tau$ along the path ${\omega}$ according to $\P$, that is  $\tau(d\omega,  dt) = \tau_{\omega} (dt) \P(d\omega)$. In the sequel, ``stopping times" including those in ${\cal T}(\mu, \nu)$, will mean possibly randomized stopping times unless stated otherwise. We note that a stopping time is {\em non-randomized} if the disintegration $\tau_{\omega}$ is the Dirac measure on $\R_+$ for $\P$ a.e. $\omega$. For further details See Section \ref{OSEP}, and also e.g. Beiglb{\"o}ck-Cox-Huesmann \cite{bch}, Guo-Tan-Touzi \cite{GTT}.

The purpose of the present article is to identify and characterize those  optimal stopping times for which (\ref{opt}) is attained. In particular, we will be investigating when optimal stopping times are unique, `true' as opposed to randomized, and whether they are hitting times of some barriers. Moreover, we consider how conditions of radial symmetry on the source and target measures are reflected in terms of symmetry in the geometry of the barrier. 
%, and are therefore unique.  
 
But first, we recall that the Skorokhod Embedding Problem  (SEP) -initiated by Skorokhod \cite{Sk65}  in the early 1960s- gives necessary and sufficient condition on a pair $(\mu, \nu)$, which insures that the set ${\cal T}(\mu, \nu)$ is non-empty.  In dimension one, this condition states that $\mu$ and $\nu$ should be in {\em convex order}, that is 
\begin{equation}
\hbox{$\int f d\mu \le \int f d\nu$ for every convex function $f$ on $\R^d$.}
\end{equation}
Since then, the problem and its variants were investigated by a large number of researchers, and have  led to several important results in  
probability theory and stochastic processes. We refer to {\OB} \cite{Ob04} for an excellent survey of the subject, which describes no less than 
 $21$ different solutions to (SEP). 
 
  More recently,  Hobson \cite{Ho98} made the connection between SEP and problems of robust pricing and hedging of financial instruments.  Hobson-Klimmek \cite{HoKl12} and Hobson-Neuberger \cite{HoNe11} connected SEP to finding robust price bounds for the forward starting straddle. See the excellent survey  of Hobson  \cite{Ho11}. 
 
The interest eventually shifted to the problem of finding optimal solutions among those that solve (SEP). In other words, which among the stopping times in  ${\cal T}(\mu, \nu)$ maximize or minimize a given cost function. 
 Among the multitude of contributions to these questions, we point out the papers of Beiglb{\"o}ck-Cox-Huesmann \cite{bch}, Cox-\OB-Touzi \cite{COT}, Dolinsky-Soner \cite{ds1, ds2}, Guo-Tan-Touzi \cite{GTT}, K{\"a}llblad-Tan-Touzi \cite{KTT}, to name just a few.
We do single out, however, the recent work of Beiglb{\"o}ck-Cox-Huesmann \cite{bch}, which uses the analogy between the optimal SEP and the theory of optimal mass transport, 
to identify and prove the so-called {\em Monotonicity Principle} (MP), which will be one of the main tools used in this paper.

 Most of the previously mentioned works focus on the one-dimensional case.  More recently, Ghoussoub-Kim-Palmer  use {dynamic programming and} PDE methods to give  solutions, {that is, optimal non-randomized stopping times},  in all dimensions for the minimization case when the cost is either $c(x, y)=|x-y|$ or when it is dynamic and given by a {certain} Lagrangian  \cite{GKP-Monge, GKP-Lagrange}.  Left open is the case of a general cost, including those of the form
\begin{align}\label{shcost}
c(x,y) = |x-y|^\alpha,
\end{align} 
where $0 < \alpha \neq 2$, or more generally cost functions of the form 
%\begin{align*} 
$ c(x,y) = f(|x-y|),$
where $f : \R_+ \to \R$ is a continuous function such that $f(0)=0$. 

 The present  paper also deals with %the somewhat related {\it optimal Skorokhod embedding problems} (OSEP) in 
higher dimensional situation, but with a focus on the case where both source and target measures are radially symmetric. 
{One of the main goals is to analyze how this symmetry impacts the structure of the barrier sets, whose hitting times determine the optimal stopping times. }
To state our first main result, we shall denote by $\P_c(O)$ the set of compactly supported probability measures whose supports are in a given open set $O$ of $\R^d$. 
 \begin{theorem}\label{thm:main} Consider the cost function $c(x,y) = |x-y|^\alpha$, where $0 < \alpha \neq 2$, and 
let $\mu, \nu \in \P_{c}(\R^d)$ be radial\footnote{A measure $\mu$ is radial if $\mu(A) = \mu(M(A))$ for any $A \in \mathcal{B}(\R^d)$ and any orthogonal matrix $M$. An example is the Gaussian measure with mean $0$ and covariance matrix $\sigma^2\,{\rm Id}$, $\sigma >0$.}, $d \ge 3$ are such that ${\cal T}(\mu, \nu)$ is nonempty.  Assume  $\mu \in H^{-1}(\R^d)$, $\mu \wedge \nu =0$,  $\mu(\{0\}) =0$, and $\mu (\partial \supp \mu)=0$.

Then, there exists a unique solution  $\tau$ for the   
 minimization/maximization problem 
\eqref{opt}, that   
 is given by a non-randomized stopping time. 
 
 Moreover, in the minimization problem with $0 < \alpha \le 1$, the solution $\tau$ is still unique also among all possibly randomized solutions, and without the assumption that $\mu \wedge \nu =0$ 
\end{theorem}

{Actually, we shall show that 
the optimal $\tau$ is a hitting time of a suitable barrier. Moreover, 
each barrier set for the Brownian motion $B^x$ starting from $x$  is axially symmetric around the line that connects the origin and $x$ (see Remark~\ref{rmk:barrier}).
}
{
A similar structure is detected in the two dimensional case, however, it was not sufficient to yield a non-randomized stopping time (see Remark~\ref{rmk:barrier-2d}). 
}

For the rest of  this introduction, we do not assume that $\mu,\nu$ are necessarily radial, but we shall need the notion of measures in subharmonic order. Recall that 
a function $f$ is subharmonic on an open set $U$ in $\R^d$ if it is upper-semicontinuous with values on $\R \cup \{-\infty\}$ and satisfies 
 $$f(x) \le \frac{1}{{\rm Leb}(B)}\int_B f(y)dy \quad \text{for every closed ball $B$ in $U$ with center at $x$}.$$
When $f$ is $C^2$, the latter condition is equivalent to $\Delta f \ge 0$ on $U$, where $\Delta=\sum_{i=1}^d \frac{\partial}{\partial x_i}$ is the Laplacian. We note that, unlike convex functions, a subharmonic function on $U$ may not have any subharmonic extension on a larger domain $V$. Therefore it is important to specify a subharmonic function with its domain. We shall therefore consider ${\cal SH}(O)$ to be the cone of all subharmonic functions on an open set $O$.
 \begin{definition}\label{subharmonicorder}
 We say that $\mu,\nu \in \P_c(\R^d)$ are in subharmonic order if the following holds:
 \begin{equation}\label{order.0}
\int f \,d\mu \le \int f \,d\nu \quad \hbox{for every} \quad f \in {\cal SH}(O),
\end{equation}
where $O$ is an open set containing $\supp (\mu+\nu)$.  
 We shall then write $\mu \prec_s \nu$.  
 \end{definition}
The fact that in higher dimension, the condition $\mu \prec_s \nu$ is also sufficient for the non-emptiness of ${\cal T}(\mu, \nu)$ has been the subject of many studies. See for example Rost \cite{Ro71}, Monroe \cite{Mo72},  Baxter-Chacon \cite{BaCh74}, Chacon-Walsh \cite{ChWa76}, Falkner \cite{Fa80} and many others.
We shall give yet another proof of this fact from a different viewpoint, and present the intimate connection between {\em the optimal  Skorokhod embedding problem} (OSEP) and {\em the subharmonic martingale optimal transport {\rm (SMOT)} problem} that we  now introduce. 

For this, recall the {\em  martingale optimal transport {\rm (MOT)}} problem. Let $\Pi(\mu,\nu)$ be the set of transport plans with {\em marginals} $\mu,\nu$, i.e. the subset of $\P(\R^d \times \R^d)$ consisting of all couplings of $\mu,\nu$ (see \cite{Vi03}). Then the MOT problem consists of the following:
\begin{align}\label{eq:MT problem}
\mbox{Minimize}\quad \text{Cost}[\pi] = \iint c(x,y)d\pi(x,y)\quad\mbox{over}\quad \pi\in {\rm MT}(\mu,\nu),
\end{align}
where MT$(\mu,\nu)$  is the set of {\it martingale transport plans,} that is, the subset of $\Pi(\mu,\nu)$ such that for each  $\pi\in$ MT$(\mu,\nu)$, its disintegration $(\pi_x)_{x\in \R^d}$  w.r.t. $\mu$ satisfies 
\begin{align}\label{eq:Jensen}
f(x) \le \int f(y)\,d\pi_x (y)
\end{align} 
for any convex function $f$ on $\R^d$. In other words, $x$ is the {\em barycenter} of $\pi_x$, $\mu$-a.e. $x$.

Problem \eqref{eq:MT problem} has been extensively studied, especially in one dimension by Beiglb{\"o}ck-Juillet \cite{BeJu16}, Beiglb{\"o}ck-Nutz-Touzi \cite{BeNuTo17}  and more recently in higher dimension by Ghoussoub-Kim-Lim  \cite{GKL2}.  For more recent developments,  see De March-Touzi \cite{DeTo17} and \OB-Siorpaes \cite{ObSi17}.
 In our situation, we need to consider the class of {\em subharmonic martingale transport plans on $O$}, that is the class 
 ${\rm SMT}_O(\mu,\nu)$  of those $\pi \in {\rm MT}(\mu, \nu)$ such that $\supp (\mu+\nu) \subset O$ and the inequality \eqref{eq:Jensen} also holds for every subharmonic function $f$ on $O$. This leads us to the {\em subharmonic  martingale optimal transport problem} (on the domain $O$): 
\begin{align}\label{eq:SMT_O problem}\mbox{Minimize}\quad \text{Cost}[\pi] = \iint_{O\times O}c(x,y)d\pi(x,y)\quad\mbox{over}\quad \pi\in {\rm SMT}_O(\mu,\nu).
\end{align}
Since every convex function is subharmonic, ${\rm SMT}_O(\mu,\nu) \subset {\rm MT}(\mu,\nu)$, and for $d=1$ these two sets are the same. However, for $d\ge 2$, the inclusion is strict and  problems \eqref{eq:MT problem} and \eqref{eq:SMT_O problem} are different. The following standard example illustrates this difference, and also the importance of the choice of the domain $O$ in the SMOT problem.
\begin{example}
Let $B_r$ be the open disk centered at $0$ with radius $r$ in $\R^2$, and let $S_r= \partial B_r$ be its boundary. Let $\rho_r$ be the uniform probability measure on $S_r$.\\
As a  first example, let $\mu=\rho_1$ and $\nu= \frac{1}{2}(\delta_0 + \rho_2)$. Then clearly ${\rm MT}(\mu,\nu) \neq \emptyset$.
On the other hand, if we consider the subharmonic function $g(z) = \log |z|$,
 then
\begin{align*}
\int g(z) \,d\mu(z) =0 > - \infty = \int g(z)\, d\nu(z),
\end{align*}
which implies that ${\rm SMT}_{\R^d}(\mu,\nu) = \emptyset$, since otherwise for any $\pi \in {\rm SMT}_{\R^d}(\mu,\nu)$,
$$\int g(y) d\nu(y) = \int\Big(\int g(y) d\pi_x(y)\Big)d\mu(x) \ge \int g(x) d\mu(x).$$

Another illustrative example consists of taking $\mu=\rho_2$ and $\nu=\rho_3$. Letting $U=\{x \ | \ 1<|x|<4 \}$ be an annulus and  $V=B_4$, then  ${\rm SMT}_{V}(\mu,\nu) \neq \emptyset$ but ${\rm SMT}_{U}(\mu,\nu) = \emptyset$. One way to see this is by considering a Brownian motion $B_t$ with $B_0 \sim \mu$, and letting $\tau_{B_3}$ be the first exit time of $B_t$ from $B_3$. Then for any $f \in {\cal SH}(V)$, $f(B_{t \wedge \tau_{B_3}})$ is a submartingale, and hence the joint law of $(B_0, B_{\tau_{B_3}})$ is an element of ${\rm SMT}_{V}(\mu,\nu)$.
On the other hand, the function $h(z) = - \log |z|$ is superharmonic on $V$ and harmonic on $U$ and satisfies $\int h d\mu > \int h d\nu$, which means that ${\rm SMT}_{U}(\mu,\nu) = \emptyset$. In terms of why a stopping time in  ${\cal T}(\mu, \nu)$   is lacking, it is because 
since $B_0 \sim \mu$, there is a positive probability that $B_t$ hits the boundary $\partial U$ before hitting $S_3$.
\end{example}
The above example shows that, unlike the MOT problem, the SMOT problem is domain sensitive, which leads  to the following natural question. Given $\mu,\nu \in \P_c(\R^d)$,  for which domain $O$, does ${\rm SMT}_{\R^d}(\mu,\nu) \neq \emptyset$ also imply that ${\rm SMT}_{O}(\mu,\nu) \neq \emptyset$?\\
To provide an answer, we make the following definition.
\begin{definition}\label{regulardomain}
We say that an open set $O$ is a regular domain  for $\mu,\nu \in \P_c(\R^d)$ if there exists a compact set $K$ in $O$ such that $\supp(\mu+\nu) \subset K$ and $\R^d \setminus K$ is connected. %In this 
We denote by ${\cal R}(\mu,\nu)$ the set of all regular domains for  $\mu,\nu$. 
\end{definition}
Now we define the set of subharmonic martingale transport plans as follows:
$${\rm SMT}(\mu,\nu)= \Cap_{O \in {\cal R}(\mu,\nu)} {\rm SMT}_O(\mu,\nu).$$
Also, define $\Psi$ to be a space of continuous test functions for subharmonic martingales
$$\Psi  :=  \{ p  \in C(\R^d \times \R^d) \, |\,  p(x,x) = 0 
 \text{ and }\, p(x, \cdot)\in {\cal SH}(\R^d)  
  \text{ for all $x \in \R^d$ 
  }\}.$$
The following theorem tells us that ${\rm SMT}(\mu,\nu)$ is not smaller than ${\rm SMT}_{\R^d}(\mu,\nu)$, and gives a precise connection between the order of $\mu,\nu$, the solvability of SMOT, and that of OSEP. 
\begin{theorem}\label{thm:main2} 
Let $\mu,\nu \in \P_c(\R^d)$. Then
\begin{align}\label{eqn:mainequality}
 {\rm SMT}(\mu,\nu)= {\rm SMT}_{\R^d}(\mu,\nu)=\Big\{\pi \in \Pi(\mu,\nu) \,\Big|\, \int p(x, y)  d\pi (x, y) \ge 0  \quad \forall p \in \Psi \Big\}
   \end{align}
and the following are equivalent:
\begin{enumerate}[1)]
\item $\mu \prec_s \nu$,
\item $\mu \prec_{s(O)} \nu$ for every $O \in {\cal R}(\mu,\nu)$,
\item  ${\rm SMT}(\mu,\nu) \neq \emptyset$,
\item ${\cal T}_O(\mu, \nu): = {\cal T}(\mu, \nu) \cap \{ \tau \ | \ \tau = \tau \wedge \tau_O\} \neq \emptyset$ for every $O \in {\cal R}(\mu,\nu)$,
\end{enumerate}
where $\tau_O$ is the first exit time of a $d$-dimensional Brownian motion $(B_t)_t$ from $O$.

Moreover, for each $\pi \in {\rm SMT}(\mu, \nu)$ and $O \in {\cal R}(\mu,\nu)$, there exists a stopping time $\tau$ such that $B_0 \sim \mu$, $B_\tau \sim \nu$, $\tau = \tau \wedge \tau_O$, and $\pi$ is the joint distribution of $(B_0, B_\tau)$.
\end{theorem}
We then explore a dual problem by considering  the class of functions
\begin{align*}
{\cal K}_c=  \{ (\alpha,\beta, p) \in C_b \times C_b \times \Psi \,|\, \beta(y) -  \alpha(x) +  p(x, y)  \le c( x, y) \},
\end{align*}
and prove the following duality result.

\begin{theorem}%[No duality gap]
\label{prop: no gap SMT}
Let $\mu,\nu \in \P_c(\R^d)$, and let $c$ be a lower-semicontinuous cost. Then the following duality holds:
\begin{align*}%\label{opt}
\sup\left\{  \int \beta \, d\nu - \int \alpha \, d\mu \ \Big| \ (\alpha,\beta, p) \in {\cal K}_c   \right\} &=\inf \left\{ \int c(x, y) \,d\pi \ \Big| \ \pi \in {\rm SMT}(\mu, \nu) \right\} 
 \\
&= \inf\left\{ \E\, [c(B_0, B_\tau)]\, | \, \tau \in {\cal T}(\mu, \nu)\right\}.
\end{align*}  
 \end{theorem}

We remark that  a dual attainment result has been achieved recently in \cite{GKP-Monge} for a certain class of cost functions, including the power costs $c(x, y)= \pm |x-y|^\alpha$, $\alpha >0$, under suitable assumptions as considered in this paper, but without  assuming radial symmetry on $\mu$ and $\nu$. We in fact use this latter result in the appendix to prove the radial monotonicity principle in Proposition~\ref{rmonoprin}, which is crucial for this paper. 

In section \ref{OSEP} we prove Theorem \ref{thm:main} by using 
that {\em radial monotonicity principle}. In section \ref{SMOT} we prove Theorems \ref{thm:main2}, and \ref{prop: no gap SMT}. 
Sections \ref{OSEP} and \ref{SMOT} can be read independently.

\section{Structure of optimal stopping times in symmetric Skorohod Embeddings}\label{OSEP}

 In this section we focus on  
  Theorem~\ref{thm:main}, and for that we start by  
  introducing a symmetric version of the monotonicity principle of Beiglb{\"o}ck-Cox-Huesmann \cite{bch}, adapted for the radially symmetric case in multi-dimensions.

We first make more precise the filtered Brownian probability space on which we operate.  We consider 
%\begin{definition} 
 $C(\R_+) = \{ \omega : \R_+ \to \R^d \,\,|\,\, \omega(0)=0, \,\omega \text{ is continuous}\}$ to be the  path space starting at $0$. The probability space will be $\Omega := C(\R_+) \times \R^d$ equipped with the probability measure $\mathbb P := \mathbb W \otimes \mu$, where $\mathbb W$ is the Wiener measure and $\mu$ is a given (initial) probability measure on $\R^d$. Stopping times and randomized stopping times will be with respect to this obviously filtered probability space.

Following 
Beiglb{\"o}ck-Cox-Huesmann \cite{bch},   let $S$ be the set of all {\em stopped paths}
\begin{align}\label{StoppedPaths}
S =\{(f^x, s) \,\,|\,\, f^x:[0,s] \to \R^d \mbox{ is continuous and $f^x(0)=x$}\},
\end{align} 
which, by letting $f^x (t ) = f^x (s)$ for $t\ge s$, can be viewed as a subset of $\Omega$.

Notice  that a (randomized) stopping time $\xi$ can be considered as a probability measure on $\Omega \times [0, +\infty]$. In particular, if the stopping time is finite almost surely, then $\xi$ is concentrated on  the set $S$ of stopped paths. In addition, $\xi$ can be disintegrated along the paths in $\Omega$ and give a family of probability measures  $\{\xi_{\omega^x}\}$ on $[0, +\infty]$, where $\omega^x (t) := \omega(t) + x$.

Now we can also interpret $\xi$ as a transport plan in the following way: for $(f^x,s) \in S$, $\xi$ transports the infinitesimal mass $d\xi \big{(}(f^x,s)\big{)}$ from $x \in \R^d$ to $f^x(s) \in \R^d$  along the path $(f^x, s)$. Now define a map $T : S \to \R^d \times \R^d$ by 
\begin{align}\label{T}
 T((f^x,s)) = (x, f^x(s))
\end{align}
and let $T\xi$ be the push-forward of the measure $\xi$ by the map $T$, thus $T\xi$ is a probability measure on $\R^d \times \R^d$, which is a transport plan. 

 Next, we introduce the \emph{conditional randomized stopping time given $(f^x,s) \in S$}, that is, the normalized stopping measure given that we followed the path $f^x$ up to time $s$. 
For $(f^x, s) \in S$ and $\omega \in C(\R_+)$, define the {\em concatenated path} $f^x\oplus\omega$ by 
\begin{align*}
 (f^x\oplus\omega)(t)  
=
 \begin{cases}
    f^x(t) \   & \text{if $t\leq s$}, \\
     f^x(s) + \omega (t-s)  & \text{if $ t>s$.}
\end{cases}
\end{align*}

\begin{definition}[Conditional stopping time \cite{bch}]\label{def:CRST}%\end{document}
 Let $\xi$ be a randomized stopping time, defined on $\Omega$. 
 The conditional randomized stopping  time of $\xi$ given $(f^x,s)\in S$, denoted by $\xi^{(f^x, s)}$,  gives a probability measure on each $\mathbb W$-a.e. $\omega \in C(\R_+)$ as follows:
\begin{align*}\label{CondStop} &\xi^{(f^x,s)}_\omega([0,t])= \frac{1}{1 - \xi_{f^x\oplus\omega}([0,s])}\left(\xi_{f^x\oplus\omega}([0,t+s])-\xi_{f^x\oplus\omega}([0,s])\right) \\ &\mbox{\,\,\,\quad\quad\quad\quad\quad \quad\quad if \quad$\xi_{f^x\oplus\omega}([0,s]) < 1$,}\\
&\xi^{(f^x,s)}_\omega(\{0\}) = 1\quad\mbox{if \quad $\xi_{f^x\oplus\omega}([0,s]) =1$.}
\end{align*}
\end{definition}  
\noindent According to \cite{bch}, this is the normalized stopping measure of the ``bush'' which follows the ``stub'' $(f^x,s)$. Note that $\xi_{f^x\oplus\omega}([0,s])$ does not depend on the bush $\omega$. 
\subsection{The Beiglb{\"o}ck-Cox-Huesmann monotonicity principle on stopped paths} 
Consider the set of all stopped paths $S$, and let  $\Gamma \subset S$ be a concentration set of a given stopping time $\tau$, i.e. $\tau(\Gamma) =1$. Given the ``stopped paths" $\Gamma$, consider the ``going paths" $\Gamma^< := \{(f, t) \,|\, \exists (g,t') \in \Gamma, t < t', \text{ and } f \equiv g \text{ on } [0,t].\}$ We shall write
\begin{align}
(x \to \psi_y : x' \to y') \in (\tau, \Gamma),  
\end{align}
if there exist $(f^{x}, t) \in \Gamma^<$, $(g^{x'}, t') \in \Gamma$  %and  {\red $t < \tau$}, 
such that  $y = f^{x}({t})$, $y' = g^{x'}({t'})$, and $\psi_y =  {\rm Law}(B^y({\tau^{(f^{x},t)}}))$.  Notice that the measure $\psi_y$ is the conditional probability measure generated by the strong Markov property of the Brownian motion that once stopped at $y$ at time $t$ then continued following the stopping time rule $\tau$. Also, the fact $(g^{x'}, t') \in \Gamma$ means intuitively that the path $g^{x'}$ drops a mass at the point $y'$ under the stopping rule $\tau$.
We shall %then 
consider the following cost for such a pair of transport
\begin{align}
C (x \to \psi_y : x' \to y') = \int c(x, z) \,d\psi_y (z) + c(x', y').
\end{align}

In the sequel, we shall denote by ${\rm SH}(x)$ the set of   %any 
probability measures $\psi$ whose barycenter $x$ satisfies $\delta_x \prec_{s} \psi$. As seen in Theorem \ref{thm:main2}, these are the probabilities that can be obtained by stopped Brownian motions starting at $x$. The following principle will be crucial for the proof ofTheorem \ref{thm:main}. It was proved  by Beiglb{\"o}ck-Cox-Huesmann \cite{bch} for more general path-dependent costs. 

\begin{theorem}[Monotonicity principle \cite{bch}]\label{monoprin} Suppose $c$ is a cost function and that $\tau$ is an optimal stopping time for the minimization problem \eqref{opt}. Then, there exists a  Borel set $\Gamma\subseteq S$ such that $\tau(\Gamma)=1$, and $\Gamma$ is {\em $c$-monotone} in the following sense: If $\psi_y \in {\rm SH}(y)$ and $(x \to \psi_y : x' \to y) \in (\tau,\Gamma)$, then
\begin{align}\label{sh5}
C(x \to \psi_y : x' \to y) \leq C(x \to y : x' \to \psi_{y}).
\end{align}
\end{theorem}
Here is a first easy but informative application of this principle, but first an important result about the case where the source and target have overlapping mass. {Recall that stopping measures (which are not necessarily probability measures) lie in $C(\R_+ ; \R^d) \times [0,+\infty]$. }

\begin{proposition} \label{stay} Let $\tau_0$ be a zero-stopping measure, i.e. concentrated on\\ $C(\R_+ ; \R^d) \times \{0\}$, and assume that its projection on the second component $\R^d$ is $\mu \wedge \nu$.   
Assume the cost function is of the form $c(x,y) = f(|x-y|)$ and that it satisfies 
\begin{equation}\label{eqn:triangle}\hbox{$c(x,x)=0$ \quad and \quad $c(x, z)\leq c(x, y) + c(y, z)$ for all $x, y, z\in \R^d$,}
\end{equation} 
with equality in the triangle inequality occurring if and only if  $x, y, z$ are along a line. 

Then for any $\tau \in {\cal T}(\mu, \nu)$ that minimizes problem \eqref{opt},  
we must have $\tau_0 \le \tau$, and therefore $\tau^*:=\tau - \tau_0$ solves the minimization problem \eqref{opt} with the same cost  
but for the disjoint source marginal $\bar \mu := \mu - \mu \wedge \nu$, and target marginal $\bar \nu := \nu - \mu \wedge \nu$.
\end{proposition}
\begin{proof} Note that having  $\tau_0 \le \tau$ means that under $\tau$, the common mass $\mu \wedge \nu$ stays put.
Let $\pi:=T\tau$ (resp., $\pi_0:=T\tau_0$) be a coupling of the pair $(\mu,\nu)$ (resp., $(\mu\wedge\nu, \mu\wedge\nu)$ 
and observe that 
\begin{align*}
\tau_0 \le \tau \iff \pi_0 \le \pi.
\end{align*}
We will prove the latter, and for this we largely follow the argument of \cite{Lim}. Let $\Gamma \subset S$ be a $c$-monotone set for $\tau$ in Theorem \ref{monoprin}. Let $T_\Gamma:=T(\Gamma)$ so that $\pi(T_\Gamma)=1$. Let $\pi_D$ be the restriction of $\pi$ on the diagonal $\Delta=\{(x,x) : x \in \R^d\}$. Note that $\supp \pi_0 \subset \Delta$. Now if $\pi_0 \nleq \pi$, the measure $\pi_0 - \pi_D$ has a non-zero positive part. Let $\eta$ be the push-forward measure of this positive part by the projection $(x,x) \mapsto x$. Denote $d \pi(x,y) = d \pi_x(y) d\mu(x)$ and $d \pi(x,y) = d \pi^y(x) d\nu(y)$ to be disintegrations of $\pi$ w.r.t. $\mu$ and $\nu$ respectively. Then by definition of $\eta$, we have
 \begin{align*}
\pi_x \neq \delta_x\, \text{ and } \, \pi^x \neq \delta_x, \quad \eta-a.e. \, x.
\end{align*}
This implies that there exist $(f^{x}, 0) \in \Gamma^<$, $(g^{z}, t) \in \Gamma$  such that  $z \neq x$, $x = g^{z}({t})$ and $\pi_x =  {\rm Law}(B^x_\tau)$. Now recall the assumption on $c$ that $c(y, x) +c(x, z) \geq c(y,z)$, and the inequality is strict unless $z,x,y$  lie on a line in this order. Hence
\begin{align*}
\int c(y, x)\, d\pi_x (y) + c(x,z) > \int c(y, z)\, d\pi_x (y).
\end{align*}
As $c(x,x)=0$, this means
\begin{align*}
C(x \to \pi_x : z \to {x}) > C(x \to x : z \to \pi_{x}),
\end{align*}
a contradiction to Theorem \ref{monoprin}.
\end{proof}
\begin{remark} 
Note that the cost $c(x, y)=|x-y|^\alpha$, $0< \alpha \le 1$ satisfies the triangle inequality \eqref{eqn:triangle}, and the above proposition implies that whenever we are studying the minimization problem with such a cost, we can assume that $\mu \wedge \nu=0$ without loss of generality.
Indeed,  the proposition says that in this case all the optimal (possibly randomized) stopping times $\tau$ of the minimization problem \eqref{opt}  are uniquely decomposed into two stopping measures as $\tau= \tau_0 + \tau^*$, where $\tau_0$ is concentrated at the time $T=0$ while $\tau^*$ is the stopping measure at positive times. Furthermore, given marginals $\mu$ and $\nu$,  $\tau_0$ is the largest possible stopping measure in such a way that $\tau_0$ acts as an identity transport from $\mu \wedge \nu$ to itself.  In  Section~\ref{sec:hitting} we will prove that  the measure $\tau^*$, once disintegrated along each path, is supported at a single time, i.e. it is non-randomized. Thus, the optimal stopping times $\tau = \tau_0+\tau^*$ can be randomized only at time $0$. 
\end{remark}

\subsection{The radially symmetric monotonicity principle}
We now give a variant of Theorem \ref{monoprin}, which exploits the radial symmetry of the marginals $\mu$ and $ \nu$. For this, we will introduce several notions related to radial symmetry. 
 First,  we give the definition of $R$-equivalence.
Recall the definition of the transport plan $T\xi$ induced by the stopping time $\xi$ (see \eqref{T}).

\begin{definition} Let $\lambda(x)=|x|$ be the modulus map. 
\begin{enumerate}
\item Two probability measures $\phi$ and $\psi$ on $\R^d$ are said to be $R$-equivalent if their push-forward measures by $\lambda$ coincide, i.e. $\lambda_\# \phi = \lambda_\# \psi$.  We then write $\phi \cong_R \psi$.
\item Two stopping times $\xi$ and $\zeta$ are said to be $R$-equivalent if the first marginals and second marginals of $T\xi$ and $T\zeta$ are $R$-equivalent, respectively. We then write $\xi \cong_R \zeta$.
\end{enumerate}
\end{definition}
\noindent The following symmetrization was introduced in \cite{Lim} for the {\em Martingale Optimal Transport Problem} with radially symmetric marginals. It will also be useful in this paper. 
Let $\mathfrak{M}$ be the group of all $d \times d$ real orthogonal matrices, and let $\mathcal{H}$ be the Haar measure\footnote{The Haar measure $\mathcal{H}$ associated to a compact topological group $G$ is the unique probability measure  which is left- and right-invariant: $H(gA) = H(A)$ and $H(Ag) = H(A)$ for every $g \in G$ and $A \in \mathcal{B}(G)$. See \cite[Chapter 11]{Fo13} for more details.} on $\mathfrak{M}$. For a given $M \in \mathfrak{M}$ and a stopping time $\xi$, we define $M\xi$ as follows: for each $A \subset S$, set
$$(M\xi)(A) = \xi(M^{-1}(A)).$$
Clearly, $M\xi$ is also a stopping time. Now we introduce the symmetrization operator which acts on both the probability measures on $\R^d$ and on the stopping times.
\begin{definition}\label{symmetrization} The symmetrization operator $\Theta$ acts on  the set of probability measures on $\R^d$, and on the set of stopping times as follows:
  \begin{enumerate}
\item  
For each probability measure $\mu$ on $\R^d$ and $B \subset \R^d$,
$$(\Theta\mu)(B) =\int_{M \in \mathfrak{M}}  (M\mu)(B) \,d\mathcal{H}(M).$$
\item  
For each stopping time $\xi$ and $A \subset S$,
$$(\Theta\xi)(A) =\int_{M \in \mathfrak{M}}  (M\xi)(A) \,d\mathcal{H}(M).$$
\end{enumerate}
\end{definition}
Observe that $\Theta\mu$ is the unique radially symmetric probability measure which is $R$-equivalent to $\mu$. Moreover, for any stopping time $\xi$, notice that (see \cite{Lim} for a proof)
\begin{align}
\text{if $T\xi$ has marginals $\mu$ and $ \nu$, then $T(\Theta\xi)$ has marginals $\Theta\mu$ and $ \Theta\nu$.}
\end{align}
This leads to the following important observation: Assume $c(x,y)$ is a rotation invariant cost function, i.e., $c(Mx, My) = c(x,y)$ for any $M \in \mathfrak{M}$,  and define the cost of a stopping time $\xi$ to be:
$$C(\xi) = \int_{\R^d \times \R^d} c(x,y)\,T\xi(dx,dy).$$
If $\xi$ solves the minimization problem \eqref{opt} where $T\xi$ has radially symmetric marginals $\mu$ and $\nu$, then for any stopping time $\zeta$, we must have
\begin{align}\label{compare}
C(\zeta) \ge C(\xi) \quad \text{whenever} \quad \zeta \cong_R \xi.
\end{align}
Indeed, if $C(\zeta) < C(\xi)$, then  since $C(\zeta)=C(\Theta\zeta)$, $\Theta\zeta$ will solve the minimization problem \eqref{opt} with the same marginals $\mu, \nu$ and less cost, a contradiction. Of course, if $\xi$ solves the maximization problem then the opposite inequality in \eqref{compare} must hold.\\

We are now ready to introduce  the radial monotonicity principle. A proof will be provided in  the appendix. 
\begin{proposition}[Radial monotonicity principle]\label{rmonoprin}  Let $c(x, y) =\pm |x-y|^\alpha$, $\alpha >0$, and $\mu, \nu \in \P_{c}(\R^d)$ be radially symmetric. 
 Assume that  $\mu \in H^{-1}(\R^d)$, $\mu\prec_s \nu$,  $\mu \wedge  \nu =0$, and $\mu (\partial \supp \mu)=0$. 
Suppose $\tau$ is an optimal stopping time for the corresponding minimization problem \eqref{opt}.  Then, there exists a  Borel set $\Gamma\subseteq S$ such that $\tau(\Gamma)=1$, and $\Gamma$ is radially $c$-monotone in the following sense: 

If $\psi_y \in {\rm SH}(y)$, $(x \to \psi_y : x' \to {y'}) \in (\tau, \Gamma)$ and $|y| = |y'|$, then
\begin{align}\label{sh6}
C(x \to \psi_y : x' \to {y'}) \leq C(x \to y : x' \to \phi_{y'}) 
\end{align}
for any $\phi_{y'} \in {\rm SH}(y')$ that is R-equivalent to $\psi_y$.

\end{proposition}

\begin{remark} The result should hold for all radially symmetric measures without the additional assumptions on $\mu$ and $\nu$, i.e.,  $\mu \in H^{-1}(\R^d)$ and  $\mu (\partial \supp \mu)=0$. 

Actually, we believe the principle of radial monotonicity should follow directly from the general monotonicity principle (Theorem~\ref{monoprin})  once applied to the case where the marginals are radially symmetric. An idea for the proof goes as follows: If the optimal stopping time $\tau$ allows a particle starting at $x$ to diffuse  when it reaches $y$ so that  it becomes the probability measure $\psi_y$, but takes another particle at ${x'}$  to stop at $y'$, and if we have the opposite inequality 
\begin{align}\label{sh2}
C(x \to \psi_y : x' \to {y'}) > C(x \to y : x' \to \phi_{y'}) 
\end{align}
for some $\phi_{y'} \in {\rm SH}(y')$, then we can  modify the stopping time $\tau$ to $\tau'$ in such a way that  the particle at $x$ now stops at $y$ by $\tau'$, but instead the particle at $x'$ starts diffusing at $y'$ until it becomes  $\phi_{y'}$. Then \eqref{sh2} means that the cost for $\tau'$ is smaller than that of $\tau$, but note that  the modified stopping time $\tau'$ may not satisfy the terminal marginal condition $\nu$. However, as $\psi_y \cong_R \phi_{y'}$ and $|y| = |y'|$, $\tau'$ is also $R$-equivalent to $\tau$, a contradiction by \eqref{compare}. The radial monotonicity principle asserts the existence of a set of stopped paths $\Gamma$ which supports the optimal stopping time and resists any such modification. %\end{proof}

\end{remark}

\subsection{A variational lemma}

 In this section we prove our key observation, namely, Lemma~\ref{key}, which shows a variational calculus with elements in ${\rm SH}(y)$. This will yield Proposition~\ref{consequence} that provides crucial comparisons that are used in the next section to  prove our main theorem.
 To explain the idea, first let
$S_{y,r}$ be the uniform probability measure on the sphere of center $y$ and radius $r$, arguably the most simple element in ${\rm SH}(y)$. We choose $c(x,y) = |x-y|$ for simplicity and consider the following ``gain" function,
\begin{align*}
G(x) = G(x,y,r) := \int  |x-z| \,dS_{y,r}(z) -  |x-y|,
\end{align*}
and its gradient $\nabla_x G$. Note that $G$ is essentially the increase in cost when the Dirac mass at $y$ diffuses uniformly onto the sphere. It satisfies the following properties:
\begin{enumerate}
\item If $|x-y| < |x'-y'|$ and $r$ is fixed, then $G(x,y,r) > G(x',y',r)$.\\
In other words, for the same diffusion, the increase in cost is greater when the distance $|x-y|$ is small. Hence, for the minimization problem, it is better to stop particles near the source $x$, and let the particles that are far from $x$ to diffuse, as long as the given marginal condition is respected.
\item The vector $\nabla G(x)$ points toward the direction $y-x$, thus the directional derivative $\nabla_u G(x)$ is
\begin{align*}
\nabla_u G(x) < 0 \quad \text{if} \quad \langle u, y-x \rangle < 0.
\end{align*}
Hence the gain function decreases when $x$ moves away from the ``center of diffusion" $y$.  By combining this with the %radial 
monotonicity principle  (Proposition~\eqref{rmonoprin}), one can get crucial information about the optimal stopping time. 
 
\end{enumerate}

From now on,  $c(x,y) = |x-y|^{\alpha}$, $0 < \alpha \neq 2$. We define the {\em gain function}
\begin{align*}
G(x, \psi_y) := \int  |x-z|^\alpha d\psi_y (z) -  |x-y|^\alpha \quad \text{for} \quad \psi_y \in {\rm SH}(y).
\end{align*}
 Note that $G$ is essentially the increase in cost when the Dirac mass at $y$ diffuses to $\psi_y$.
The following variational lemma is crucial for our analysis.
\begin{lemma}\label{key}
Let $x, y$ be nonzero vectors in $\R^d$ and let $r = |x|$. Let $u$ be a unit tangent vector to the sphere $S_r$ at $x$, such that $\langle u, y \rangle <0$. Let $\psi_y \in {\rm SH}(y) \setminus \{ \delta_y\}$. Then, there exists a $\phi_y \in {\rm SH}(y)$ which is R-equivalent to $\psi_y$, such that
\begin{align*}
G(x, \phi_y) &= G(x, \psi_y) \\
\nabla_u G(x, \phi_y) &< 0 \quad \text{if} \quad \,0 < \alpha < 2,\\
\nabla_u G(x, \phi_y) &> 0 \quad \text{if} \quad\quad\quad \alpha > 2,
\end{align*}
where the directional derivative is applied to the $x$ variable.
\end{lemma}
\begin{proof} 
Let $c(x,z) = |x-z|^{\alpha}$ and take its partial derivative $\frac{\partial}{\partial {x_d}}$ at $x=0$, to obtain
\begin{align}\label{sh9}
h(z) := \nabla_{e_d} c (x,z) \big|_{x=0} = - \alpha \, |z|^{\alpha - 2}\, z_d
\end{align}
Taking the Laplacian in $z$, we get
\begin{align}\label{laplacian}
\Delta \,h(z) = -\alpha (\alpha - 2) (\alpha + d - 2) \,|z|^{\alpha -4} \, z_d.
\end{align}
We see that the function $h$ is
\begin{enumerate}[i)]
\item strictly superharmonic in the lower half-space $\{z_d < 0\}$ if $0 < \alpha <2$
\item strictly subharmonic in the lower half-space $\{z_d < 0\}$ if  $\alpha >2.$
 \end{enumerate}
Let $0 < \alpha < 2$, and assume without loss of generality that $x=x_1\, e_1=(x_1,0,...,0)$ and $u = e_d$. Then $\langle u, y \rangle <0$, which means that $y$ is in the lower half-space. Let $H=\{ z \in \R^d : z_d = 0\}$ and choose a stopping time $\tau$ for the Brownian motion $B^y$ such that $ {\rm Law}(B^y_\tau) = \psi_y$, and let $\eta$ be the first time $B^y$  hits $H$. Let $\sigma_y =  {\rm Law}(B^y_{\tau \wedge \eta})$. Then $\sigma_y$ is supported in the closed lower half-space and it is nontrivial, hence by the strict superharmonicity \eqref{laplacian}, we have
\begin{align*}
\nabla_u G(x, \sigma_y) < 0.
\end{align*}

Now we modify $\tau$ to $\tau'$ in the following way; if $\tau \leq \eta$, then we let $\tau' = \tau$. But if $\tau > \eta$, in other words if a particle at $y$ has landed on $H$ but not completely stopped by $\tau$, then we symmetrize the remaining time of $\tau$ (i.e. the conditional stopping time) with respect to $H$ and get $\tau'$. More precisely, let $\tau_H$ be the reflection of the  conditional stopping time of $\tau$ with respect to  $H$; that is, if $\tau$ stops a path emanating from $H$, then $\tau_H$ stops the reflected path at the same time. Now define $\tau' := \frac{\tau+\tau_H}{2}$ to be a randomization; before re-starting Brownian motion on $H$, we flip a coin and choose either $\tau$ or $\tau_H$ for the conditional stopping time.

Now, define $\phi_y =  {\rm Law}(B^y_{\tau'})$ and observe that
\begin{enumerate}[i)]
\item $G(x, \phi_y) = G(x, \psi_y)$ and $\phi_y \cong_R \psi_y$, by the symmetry with respect to $H$ in the definition of $\tau'$.

\item $\nabla_u G(x, \phi_y) = \nabla_u G(x, \sigma_y)$, since the function $z \mapsto \nabla_{e_d} c (x,z)$ is odd in $z_d$ (see \eqref{sh9}) hence the symmetrization in the definition of $\tau'$ does not change $\nabla_u G$.
\end{enumerate}
This proves the lemma.
\end{proof}
 Notice that the above lemma  
 implies, in particular for $0< \alpha <2$, that
\begin{align}\label{eq:infinitesimal}\nonumber
0 & <  G(x, \psi_y) - G(z, \phi_{y}) = C(x \to \psi_y : z \to {y}) - C(x \to y : z \to \phi_{y}) \\\nonumber
& \text{if $z = x + \epsilon u $ with $\langle u, y \rangle <0$,   for small $\epsilon >0$, where $\epsilon$ depends on $x$ and $\psi_y$.}
\end{align}
In the following proposition, we make this result less infinitesimal. This will be crucial in the next section to prove our main theorem.
\begin{proposition}\label{consequence}
Let $\phi_y \in {\rm SH}(y) \setminus \{\delta_y\}$ be given, and define
\begin{align*}
\underline{G} (x) &:= \displaystyle\min_{\sigma_y \in {\rm SH}(y), \sigma_y \cong_R \phi_y}  \int  |x-z|^\alpha d\sigma_y (z) - |x - y|^\alpha,\\
\overline{G} (x) &:= \displaystyle\max_{\sigma_y \in {\rm SH}(y), \sigma_y \cong_R \phi_y}  \int  |x-z|^\alpha d\sigma_y (z) - |x - y|^\alpha.
\end{align*}
Then $\underline{G}, \overline{G}$ are attained, and for $x_0, x_1, y \in \R^d$ with $|x_0| = |x_1|$ and $|x_0 - y| < |x_1 - y|$,
\begin{align*}
&\underline{G} (x_0) > \underline{G} (x_1) \quad \text{and} \quad \overline{G} (x_0) > \overline{G} (x_1)   \quad \text{if} \quad 0 < \alpha <2,\\
&\underline{G} (x_0) < \underline{G} (x_1) \quad \text{and} \quad \overline{G} (x_0) < \overline{G} (x_1)   \quad \text{if} \quad\quad\,\,\,\, \, \alpha > 2.
\end{align*}
\end{proposition}
\begin{proof}
Define
\begin{align*}
 \underline{\Re}(x, \phi_y) &:= \bigg{\{}\psi_y \in {\rm SH}(y) : \psi_y \in \arg\displaystyle\min_{\sigma_y \cong_R \phi_y}  \int  |x-z|^\alpha d\sigma_y (z)\bigg{\}}, \\
 \overline{\Re}(x, \phi_y) &:= \bigg{\{}\psi_y \in {\rm SH}(y) : \psi_y \in \arg\displaystyle\max_{\sigma_y \cong_R \phi_y}  \int  |x-z|^\alpha d\sigma_y (z)\bigg{\}}.
 \end{align*}
We will soon show that these are nonempty. Note that, by definition
\begin{align*}
\underline{G} (x) &= G (x, \psi_y) =  \int  |x-z|^\alpha d\psi_y (z) - |x - y|^\alpha \quad \text{for any} \quad \psi_y \in \underline{\Re}(x, \phi_y),\\
\overline{G} (x) &= G (x, \psi_y) =  \int  |x-z|^\alpha d\psi_y (z) - |x - y|^\alpha \quad \text{for any} \quad \psi_y \in \overline{\Re}(x, \phi_y).
\end{align*}
First, we claim that $\underline{G}(x)$ and $\overline{G}(x)$ are continuous.\\
Indeed, let $x_n \rightarrow x$ in $\R^d$, and define
\begin{align*}
\underline{C} (x) := \underline{G} (x) + |x-y|^\alpha =  \int  |x-z|^\alpha d\psi_y (z)\quad \text{for any}\quad \psi_y \in \underline{\Re}(x, \phi_y).
\end{align*}
Set $a_n = \underline{C} (x_n)$ and $a = \underline{C} (x)$. 
Choose any subsequence $\{a_k\}$ of $\{a_n\}$ and  a corresponding sequence of measures $\psi_k \in \underline{\Re}(x_k, \phi_y) $. By compactness, $\{\psi_k\}$ has a subsequence $\{\psi_l\}$ which converges to, say $\psi$. Note that $\psi \in {\rm SH}(y)$ and $\psi \cong_R \phi_y$ by weak convergence. Now, write
\begin{align*}
\int  |x_l-z|^\alpha d\psi_l (z) - \int  |x - z|^\alpha d\psi (z) 
&= \bigg[ \int  (|x_l-z|^\alpha - |x - z|^\alpha) \,d\psi_l (z) \bigg] \\
&\quad + \bigg[\int  |x - z|^\alpha d\psi_l (z) - \int  |x - z|^\alpha d\psi (z) \bigg].
\end{align*}
The first bracket goes to zero as $l \rightarrow \infty$ since $|x_l-z|^\alpha - |x - z|^\alpha \rightarrow 0$ uniformly on every compact set in $\R^d$, and the second bracket goes to zero since $\psi_l \rightarrow \psi$.\\
Now we claim that 
\begin{align}
\int  |x - z|^\alpha d\psi (z) = \underline{C} (x) = a, \quad \text{i.e.}\quad \psi \in \underline{\Re}(x, \phi_y).
\end{align}
If not, then there exists a $\rho \in \underline{\Re}(x, \phi_y)$ such that 
\begin{align*}
\int  |x - z|^\alpha d\psi (z) &> \int  |x - z|^\alpha d\rho (z), \text{ hence} \\ 
\int  |x_l - z|^\alpha d\psi_l (z) &> \int  |x_l - z|^\alpha d\rho (z) \text{ \,for all large $l$,}
\end{align*}
a contradiction since $\psi_l \in \underline{\Re}(x_l, \phi_y)$. Therefore, $a_n \rightarrow a$, since this holds for any subsequence. 

To complete the proof of the proposition, we let $0 < \alpha < 2$. Note that $x_0, x_1, y$ are nonzero. Let $r= |x_0|=|x_1|$. Without loss of generality, we can assume that there exists a differentiable curve $x(t): [0, 1] \rightarrow S_r$ with $x(0) = x_0, \,x(1) = x_1$ such that $|\,x(t) - y\,|$ is strictly increasing. In other words, $x(t)$ satisfies that $|x(t)|=r$ and $\langle \frac{d}{dt}x(t), y \rangle <0$ for all $t$. We note that, although finding such a curve is not always possible, we can always choose an alternative point $x_1' = Mx_1$ for some orthogonal matrix $M$ with $My=y$, so that there is a geodesic curve on the sphere $S_r$ connecting $x_0$ and $x_1'$ with the desired property. Notice that $\underline{G}(x_1) = \underline{G}(x_1')$, hence such a change does not affect the conclusion.

Now for a fixed $t \in [0,1]$, choose any $\psi_y \in \underline{\Re}(x(t), \phi_y) $. Then Lemma \ref{key} gives $\sigma_y \in \underline{\Re}(x(t), \phi_y) $ with $\frac{d}{dt} G(x(t), \sigma_y) < 0$. By definition, $\underline{G}(x(s)) \leq G(x(s),\sigma_y)$ for any $s$, and $\underline{G}(x(t)) = G(x(t), \sigma_y)$. Hence
\begin{align*}
\limsup_{\epsilon \downarrow 0} \frac{\underline{G}(x(t+\epsilon)) - \underline{G}(x(t))}{\epsilon} &\leq 
\limsup_{\epsilon \downarrow 0} \frac{G(x(t+\epsilon), \sigma_y) - G(x(t), \sigma_y)}{\epsilon} \\
&= \frac{d}{dt} G(x(t), \sigma_y) < 0.
\end{align*}
The function $\underline{G} (x(t))$ is continuous and satisfies the above strict inequality for each $t \in [0,1]$, hence it must be strictly decreasing. \\
For $\overline{G} (x(t))$, we similarly use $\sigma_y \in \overline{\Re}(x(t), \phi_y) $ and $\frac{d}{dt} G(x(t), \sigma_y) < 0$ to get
\begin{align*}
\liminf_{\epsilon \downarrow 0} \frac{\overline{G}(x(t - \epsilon)) - \overline{G}(x(t))}{\epsilon} &\geq 
\liminf_{\epsilon \downarrow 0} \frac{G(x(t - \epsilon), \sigma_y) - G(x(t), \sigma_y)}{\epsilon} \\
&= -  \frac{d}{dt} G(x(t), \sigma_y) > 0.
\end{align*}
We again see that $\overline{G} (x(t))$ is strictly decreasing. 

The case $\alpha >2$ can be proved in a similar fashion, and the proposition follows.
\end{proof}

\subsection{Optimal Stopping problem for radially symmetric marginals}\label{sec:hitting}

Finally, armed with Lemma~\ref{key} and 
Proposition \ref{consequence},   we establish   Theorem \ref{thm:main}. 
\begin{proof}[\bf Proof of Theorem \ref{thm:main}]   We give a proof for the minimization problem for the case $0< \alpha <2$;  the case $\alpha > 2$, or the maximization problem, 
can be proved similarly.

Fix $0< \alpha<2$, and let $\tau$ be a minimizer for  \eqref{opt}. 
%{\red We first symmetrize $\tau$ to $\bar \tau$. }
For $x, y \neq 0$ in $\R^d$ with $x \not\parallel y$, (here $\parallel$ denote parallelism), we define the barrier set:
$$U_x^y = \{z \in \R^d : |z| = |y| \text{ and } \langle x, z \rangle > \langle x, y \rangle \}.$$ 
 The set $U_x^y$ looks like the spherical cap of radius $|y|$, which is symmetric around the axis in the direction of $x$, containing the point $y$ in its boundary.

We shall say that a pair $(f, s)$ and $(g, t)$ in  $S$ is {\em forbidden} if  
$$ 
f(0)=g(0) \neq 0 \text{ and } \exists s' < s \text{ such that } f(s') \in U_{g(0)}^{g(t)}.
$$
In words, a forbidden pair consists of a path that penetrates 
 the barrier generated by the other path.
We let ${\bf FP}$ denote the set of forbidden pairs. %The proof consists of two parts:\\

 First, we show that there exists  $\Gamma \subset S$ on which $\tau$ is concentrated, such that $\Gamma$ does not admit any ``forbidden pair" that lies in $ \Gamma \times \Gamma$. 
 Indeed, choose the  $c$-monotone $\Gamma$ for $\tau$ as given by  Proposition~\ref{rmonoprin} and suppose that ${\bf FP} \cap (\Gamma \times \Gamma) \neq \emptyset$,  i.e. there exists a forbidden pair $(f, s)$ and $ (g,t)$ in $\Gamma$ where $(f,s)$ penetrates the barrier $U_{g(0)}^{g(t)}$, that is, not stopping when it hits the set.  Then, by the Markov property, the penetrating path $(f, s)$  yields a nontrivial subharmonic measure,  say $\psi_y \in {\rm SH} (y)$, namely the conditional probability, whose barycenter is at the point $y=f^x(s')$ where the barrier $U_{g(0)}^{g(t)}$ is hit. But this contradicts Proposition \ref{consequence} and Proposition ~\ref{rmonoprin}.  Hence, ${\bf FP} \cap (\Gamma \times \Gamma) = \emptyset$. 
 
 Now, we show that  since $\Gamma$ does not allow forbidden pairs, then every stopping time concentrated on $\Gamma$ is necessarily non-randomized, which clearly yields the uniqueness. Indeed, let $\xi$ be any stopping time in ${\cal T}(\mu, \nu)$ with $\xi(\Gamma) = 1$. Define $\Gamma_0 = \{(f^x, s)\in \Gamma \ | \ s=0\}$, i.e. $\Gamma_0$ consists of the paths that are stopped immediately at time $0$. Let $\Gamma_+ = \Gamma \setminus \Gamma_0$. We can assume $x \neq 0$ for every $(f^x, s) \in \Gamma$ as $\mu (\{0\}) =0$. Moreover, because $d\ge 3$, the probability of the Brownian motion from $x$ to hit  a line segment is zero, so we can assume that $f^x(s) \not \parallel x$.\footnote{The assumption $d\ge 3$ is used here only.} Let $\xi_0$ be the restriction of $\xi$ on $\Gamma_0$, and $\xi_+$ be on $\Gamma_+$.

We claim that since ${\bf FP} \cap (\Gamma \times \Gamma) = \emptyset$,  $\xi_+$ must be of non-randomized type. \\
Let us suppose the contrary. Then there exists an element $(f^x, s) \in \Gamma$, $ s>0, x\ne 0$,  such that the conditional stopping time $\xi^{(f^x, s)}$ is nonzero. 
This means that the Brownian motion which has followed the path $f^x$ up to time $s>0$ will continue its motion at $y := f^x(s)$. Now consider the barrier $U := U_x^y$ and note that the Brownian motion starting at $y$ governed by any non-zero stopping time will go through the surface $U$ before its complete stop, since $y$ is on the boundary  of $(d-1)$-dimensional set $U$. This implies that there is a stopped path $(g^y, t) \in S$ such that the concatenation $(f^x\oplus g^y, s+t) \in \Gamma$ and for some $s' < s+t$, $(f^x\oplus g^y) (s') \in U$. This means that the pair $((f^x\oplus g^y, s+t), (f^x, s)) \in \Gamma \times \Gamma$ is a forbidden pair, which is a contradiction. 

The separation assumption $\mu \wedge \nu=0$ implies $\xi_0=0$, yielding that $\xi$ is non-randomized. This implies the uniqueness of $\xi$ in the usual way, that is if $\tau$ and $\tau'$ are two minimizers and if  their disintegrations do not agree, i.e. $\tau_{\omega^x} \neq \tau'_{\omega^x}$ for all $\omega^x \in B$ with $\P(B) >0$, then the stopping time  $\frac{{\tau} + {\tau'}}{2}$, which is obviously a minimizer, must be of randomized type, thus  yielding a contradiction. This concludes the proof. 
 \end{proof}
\begin{remark}
In fact the proof of Theorem \ref{thm:main} shows that, if the radial monotonicity principle holds without the assumption $\mu \wedge \nu=0$, then every optimal stopping time $\tau$ is  decomposed into two stopping measures as $\tau= \tau_0 + \tau^*$, where $\tau_0$ is supported  at time $T=0$ while $\tau^*$ is supported in $T >0$ and is non-randomized. %Note that this implies the uniqueness of $\tau$ under the assumption $\mu \wedge \nu=0$, as this forces $\tau_0=0$.
\end{remark}

\begin{remark}\label{rmk:barrier}
 Let $\Gamma$ be the $c$-monotone set as given in Proposition~\ref{rmonoprin} on which the optimizer in Theorem \ref{thm:main} is concentrated. The proof of Theorem \ref{thm:main} in fact tells us that, for the minimization problem with  $0 < \alpha < 2$ or the maximization problem with $\alpha > 2$, 
 the optimal stopping time  is given by the first time Brownian motion $B^x$ hits the following union of barriers
$$\mathcal{U}_x := \cup_y U_x^y,  \text{ where } y = f^x(s), \text{$y\not\parallel x$}, \text{ for some } (f^x, s) \in \Gamma. $$
Moreover, by the uniqueness property and the radial symmetry of $\mu$ and $\nu$, the sets $\mathcal{U}_x$'s  are congruent under rotation, that is if $M$ is an orthogonal matrix and $M(x) = x'$, then $M(\mathcal{U}_x) = \mathcal{U}_{x'}$.

For minimization problem with $\alpha > 2$
or maximization problem with $0 < \alpha < 2$, 
we have the same type of result, but the barrier will be reversed: it will be given by
$$\mathcal{V}_x := \cup_y V_x^y, \text{ where } y = f^x(s), \text{$y\not\parallel x$,}  \text{ for some } (f^x, s) \in \Gamma, $$
where $V_x^y$ is the reversed barrier
$$V_x^y = \{z \in \R^d : |z| = |y| \text{ and } \langle x, z\rangle < \langle x, y \rangle \}.$$ 

This is due to the interchange of the superharmonic and subharmonic region of the derivative of the gain function \eqref{sh9}, according to the value of $\alpha$. Also note that when dealing with  maximization problem, the inequalities \eqref{sh5} and \eqref{sh6} 
%{\red \eqref{eq:reflection monotone}}
in the monotonicity principles must be reversed.
\end{remark}

\begin{remark}\label{rmk:barrier-2d}
 Observe that in the above proof, the argument breaks down in the two dimensional case, because the probability for a Brownian motion to hit a line segment is not zero when the dimension $\le 2$.  Still, it shows that the optimal (possibly randomized) stopping time $\tau$ has to stop the Brownian path completely once it stops at a point not parallel to the initial point. Therefore, in particular, for the minimization problem with $0<\alpha <2$, the set 
\begin{align*}
&\mathcal{U}_x := \cup_y U_x^y, \text{ where } y = f^x(s),  y\not\parallel x,   \text{ for some } (f^x, s) \in \Gamma \quad  
\end{align*}
 is a barrier. However, there are chances that the Brownian path may stop only partially at points parallel to $x$, then to continue until it hits the set $\mathcal{U}_x $. %
Similar result holds for the maximization problem or with $\alpha>2$.  
\end{remark}

 Finally, we note that the ideas in this section can be applied to costs that are more general than the ones of the form $|x-y|^\alpha$ considered in this paper. In particular, they apply to a class of cost functions $c(x,y)$ that are invariant under rotation and whose directional derivatives $ \nabla_u c(x,y)$ in the $x$-variable have suitable sub/superharmonic regions in the $y$-variable. 

 \section{Subharmonic martingale optimal transport problem}% and its connection with Skorokhod embedding problem}
 \label{SMOT}
 In the previous section we focused on the structure or optimal stopping times, utilizing the radial symmetry of marginals. From now on we do not assume the radial symmetry but consider general marginals. We assume compact support, but see Proposition \ref{prop1} for an exception.

As mentioned in the introduction, a subharmonic function on an open subset of $\R^d$ may not allow a subharmonic extension on all of $\R^d$ in general. When can $f \in {\cal SH}(O)$ be approximated by $\tilde f \in {\cal SH}(\R^d)$ in $O$? To give an answer, we recall the following. %Walsh's theorem.
\begin{lemma}[Walsh \cite{W29}] Let $K$ be a compact subset of $\R^d$ such that $\R^d \setminus K$ is connected. Then for each $\eps >0$ and a harmonic function $u$ on an open set containing $K$, there exists a harmonic polynomial $v$ such that $|u-v| < \eps$ on $K$.
\end{lemma}
Keeping this in mind, we turn to the proof of Theorem \ref{thm:main2}. First we need the following very likely known lemma.
Given an open set $O$ in $\R^d$, let $C(O)$ be the space of continuous functions on $O$. We give a topology on $C(O)$ which is induced by the following convergence, namely, $f_n \to f$ iff $||f_n-f||_{L^\infty(K)} \to 0$ on every compact subset $K$ of $O$. It is clear that this topology is metrizable via the following metric
$$d(f,g) := \sum_{n=1}^\infty 2^{-n}\min(1, ||f-g||_{K_n})$$
where $\{K_n\}$ is a compact exhaustion of $O$; $K_n$ is compact, $K_n \subset {\rm int} (K_{n+1})$, and $\cup_n K_n = O$.
\begin{lemma}\label{separable}
The space ${\cal H}(O)$ of harmonic functions on $O$, is separable under $d$.
\end{lemma}
\begin{proof}
The space $C(O)$ is separable under the metric $d$, e.g. the set ${\displaystyle \mathbb {Q} [x_1,...,x_d]}$ of polynomials with rational coefficients is a countable dense subset of $C(O)$ by Stone-Weierstrass Theorem. Then since every subspace of a separable metric space is separable, ${\cal H}(O)$ is separable as well.
\end{proof}

\begin{proof}[Proof of Theorem \ref{thm:main2}]

To prove \eqref{eqn:mainequality}), note that 
if  $\pi \in {\rm SMT}_{\R^d}(\mu,\nu)$, then for any $p \in \Psi$ we have 
\begin{align*}
 \int p(x, y) d\pi (x, y) = \int p (x, y) d\pi_x (y) d\mu(x)  \ge  \int  p(x, x) d\mu (x) =0
 \end{align*}
in view of  the subharmonicity of  of $y \mapsto  p(x, y)$. Hence we have the following inclusions
\begin{align*}
 {\rm SMT}(\mu,\nu) \subset {\rm SMT}_{\R^d}(\mu,\nu) \subset \Big\{\pi \in \Pi(\mu,\nu) \,\Big|\, \int p(x, y)  d\pi (x, y) \ge 0  \quad \forall p \in \Psi \Big\}.
   \end{align*}
Now let $\pi \in \Pi(\mu,\nu)$ and assume $ \int p \,d\pi \ge 0$ for every $ p \in \Psi$.  We now prove that $\pi \in {\rm SMT}(\mu, \nu)$. Indeed, let $\mathbf{1}_{B_{z,r}}$ be the indicator function on the ball $B_{z,r}$ of center $z$ and radius $r$, and consider the functions of the form 
\begin{align*}
 p(x, y) = \mathbf{1}_{B_{z,r}}(x)(\phi (y) - \phi(x)) \quad \text{where} \quad \phi \in {\cal H}(\R^d).
\end{align*}
Then $p(x, y) \in \Psi$\footnote{Technically $\mathbf{1}_{B_{z,r}}(x)$ is not continuous, but it can be approximated by continuous and compactly supported functions.}. For $z \in \supp \mu$, we have
\begin{align*}
 0 &\le \frac{1}{\mu(B_{z,r})}\iint p(x, y) d\pi (x, y) = \frac{1}{\mu(B_{z,r})}\int_{B_{z,r}} \left(\int \phi (y) d\pi_x (y)  - \phi (x) \right) d\mu(x) \nonumber \\
 & \to  \int \phi (y) d\pi_z (y)  - \phi (z) \quad \text{as} \quad r \to 0, \quad \mu - a.e.\, z,
\end{align*}
where the convergence is due to e.g. \cite[Lemma 4.1.2.]{LY85}. By changing $\phi \mapsto -\phi$ in $p(x,y)$, we deduce
\begin{align}\label{harmoniceq}
\int \phi (y) d\pi_z (y) = \phi (z) \quad \mu - a.e.\, z.
\end{align}
One may notice that the $\mu$ - a.e. convergence set could depend on the choice of $\phi$, but the application of Lemma \ref{separable} then ensures that the equality \eqref{harmoniceq} holds for all $\phi \in {\cal H}(\R^d)$, $\mu$ - a.e..

Let $\psi$ be the fundamental solution of the Laplace equation, and  %as in Lemma \ref{approximation}.
 let $\psi_{a}(x) =\psi(x-a)$ and $\psi_{a,c}(x) = \max(\psi(x-a) , c)$ where $a \in \R^d$ and $c \in \R$. Note that $\psi_{a,c}$ is continuous and subharmonic on $\R^d$. By applying the above argument with $p(x,y)= \mathbf{1}_{B_{z,r}}(x)(\psi_{a,c}(y) - \psi_{a,c}(x))$ and letting $c \searrow -\infty$, we deduce 
\begin{align}\label{kernalineq}
\int \psi_a(y) d\pi_z (y) \ge \psi_a(z) \quad\text{for every } a \in \R^d,  \,\,\mu - a.e.\, z.
\end{align}
Now let $O$ be a regular domain for $\mu,\nu$, and let $K$ be a compact set in $O$ such that $\supp(\mu+\nu) \subset K$ and $\R^d \setminus K$ is connected. Let $O'$ be an open precompact subset of $O$ containing $K$. For each $f \in {\cal SH}(O)$, there exists a nonnegative Borel measure $\kappa$ on $O$ such that $f$ can be decomposed as
$$f(x) = h(x) + \int_{O'}\psi_a(x) d\kappa(a) \quad \forall x \in O', \text{ for some } h \in {\cal H}(O')$$
by Riesz representation theorem. Let $h_\eps$ be a harmonic polynomial such that $|h-h_\eps| < \eps$ on $K$ by Walsh's theorem. Then for $\mu$-a.e. $x$, we have
\begin{align*}
\int f(y)d\pi_x(y) &= \int h(y) d\pi_x(y)+ \int_{\R^d}\int_{O'}\psi_a(y) d\kappa(a)d\pi_x(y) \\
& \ge \int h_\eps (y) d\pi_x(y) + \int_{O'}\int_{\R^d}\psi_a(y) d\pi_x(y) d\kappa(a) -\eps  \\
&\ge h_\eps(x) + \int_{O'}\psi_a(x) d\kappa(a) -\eps\\
& \ge h(x)+ \int_{O'}\psi_a(x) d\kappa(a) -2\eps = f(x) -2 \eps.
\end{align*}
Letting $\eps \to 0$ we get $\int f(y)d\pi_x(y)  \ge f(x)$ $\mu$-a.e. $x$, implying $\pi \in {\rm SMT}_O(\mu,\nu)$. As $O \in {\cal R}(\mu,\nu)$ was arbitrary we deduce that $\pi \in {\rm SMT}(\mu,\nu)$. This completes the proof of the equality \eqref{eqn:mainequality}.

Now we prove the equivalence in Theorem \ref{thm:main2}. Notice that $(1) \iff (2)$ is immediate by the same approximation argument as above. 
The direction $(3) \implies (1)$ is also immediate, as for $f \in {\cal SH}(\R^d)$ and $\pi \in {\rm SMT}(\mu,\nu)$,
\begin{align*}
 \int f(y) d\nu (y) = \int f(y) d\pi_x (y) d\mu(x)  \ge  \int  f(x) d\mu (x).
 \end{align*}
To see that $(4) \implies (3)$, let $O$ be an open ball containing $\supp( \mu+\nu)$, and take $\tau \in {\cal T}_O(\mu,\nu)$. Then for $f \in {\cal SH}(\R^d)$, since $(f(B_{t \wedge\tau}))_t$ is a submartingale, we have 
$$\E[f(Y) \ | \ X] \ge f(X), \quad \text{where} \quad X=B_0 \sim \mu, \,\,Y=B_\tau \sim \nu.$$
 This means that the joint law of $(X,Y)$  belongs to ${\rm SMT}(\mu,\nu)$.

It remains to show the implication $(2) \implies (4)$. This will be done in Proposition~\ref{prop1} below, where we prove it for possibly non-compactly supported marginals.  

%{\bf Step 3:}  
Finally, let now $\pi \in {\rm SMT}(\mu,\nu)$ and its disintegration $(\pi_x)_x$ with respect to $\mu$. Let $(B_t)_t$ be a Brownian motion with $B_0 \sim \mu$, and let $O \in {\cal R}(\mu,\nu)$ be bounded.  Since $\delta_x\prec_{s(O)} \pi_x$ for $\mu$-almost all $x$, and noting that $B_0$ and $B_t - B_0$ are independent, one can apply the implication $(2) \Rightarrow (4)$ to the subharmonic ordered pair  $(\delta_{x}, \pi_{x})$, and use the measurable selection theorem, to find %select measurably 
 a stopping time $\tau=\tau\wedge \tau_O$ such that given $B_0$, we have ${\rm Law}(B_\tau \in \, \cdot \, \,|\, B_0) = \pi_{B_0} (\, \cdot \,)$. It is then clear that ${\rm Law} (B_0, B_\tau) = \pi$. 
  \end{proof}
The following is now immediate.

\begin{corollary} \label{same} Let $\mu,\nu \in \P_c(\R^d)$  and assume $\mu \prec_s \nu$. Then
\begin{align*}%\label{opt}
 \inf\left\{ \E\, [c(B_0, B_\tau)]\, | \, \tau \in {\cal T}(\mu, \nu)\right\}=\inf\left\{\iint c(x,y)d\pi \, | \, \pi\in {\rm SMT}(\mu,\nu)\right\}.
 \end{align*}
\end{corollary}
The following proposition may have its own interest. Let $\P_2(O)$ be the set of probability measures concentrated in $O$ and with finite second moments. Define the order $\mu \prec_{s_2(O)} \nu$ by
$$\mu \prec_{s_2(O)} \nu \iff \int f d\mu \le \int f d\nu \text{\,\, for every $f\in {\cal SH}(O)$ with } f(x) \le C_f (1+|x|^2).$$
\begin{proposition}\label{prop1} 
Let $\mu, \nu \in \P_2(O)$ and assume $\mu \prec_{s_2(O)} \nu$. Then $\cal{T}_O(\mu,\nu) \neq \emptyset$.
\end{proposition}
\begin{remark} For the proposition, we need not assume $O$ is regular since the proof clearly implies that if $\mu, \nu \in \P_c(O)$ are such that $\mu \prec_{s(O)} \nu$, then necessarily $\cal{T}_O(\mu,\nu) \neq \emptyset$.

\end{remark}
To prove the proposition, we first introduce the notion of {\em spherical martingales}.
\begin{definition}
Let $U$ be the uniform probability measure on the unit sphere  $S^{d-1}$ in $\R^d$. Let $(X_i)^{\infty}_{i=1}$ be i.i.d random variables on some probability space $(\Omega, \P)$, whose distribution is $U$. A stochastic process $(F_n)_{n \ge 0}$ is called a {\em spherical martingale (valued) in $O$} if there is an associated sequence  of bounded measurable functions 
$$r_n:  \R^d \times \underbrace{ S^{d-1}\times \cdots \times S^{d-1} }_{(n-1)-times} \to \R_+ $$ 
such that 
\begin{align*}
%&F_0 = x \in \R^d\\
F_n (X_1,...,X_n) - F_{n-1} (X_1,...,X_{n-1}) = r_n (F_0, X_1,...,X_{n-1}) \cdot X_n
\end{align*}
and if for every $n \in \N$, $0 \le \lambda \le 1$ and $(X_1,X_2,...,X_n)$,
\begin{align*}
F_{n-1} (X_1,...,X_{n-1}) +  \lambda r_n (F_0, X_1,...,X_{n-1}) X_n \in O.
\end{align*}
\end{definition}
 At each time $n$, a particle splits uniformly onto a surrounding sphere of radius $r_n$.  An important observation is that the  push-forward measure of $\delta_x$ by a spherical martingale $F_n$  has the same law as $B^x_{\tau}$, where the stopping time $\tau$ is defined as follows:  
$\tau_1$ is the first time $B^x$ hits the sphere $S(x, r_1 (x))$ centered at $x$ and with radius $r_1(x)$. If $B_{\tau_1} (\omega) = x_1 \in S(x,  r_1(x))$, then define $\tau_2$ to be the first hitting time  $B^{x_1}$ hits the sphere $S(x_1, r_2 (x, x_1))$. One can then define inductively a sequence of stopping times $\tau_1 \leq \tau_2 \leq ...$ such that $F_n (\P) = {\rm Law}(B^x_{\tau_n})$.

The following lemma is an analogue of a result of Bu-Schachermayer \cite[Proposition 2.1]{BuSc92} about analytic martingales.
\begin{lemma}\label{shenvelope}
%{\red 
Let $O$ be an open set in $\R^d$ and $f : O \rightarrow \R \cup \{-\infty\}$ be an upper semicontinuous function. Then there exists a unique maximal subharmonic function $\hat f$ on $O$ which is dominated by $f$. In other words, $\hat f \le f$ and $g \le \hat f$ for every $g \in {\cal SH}(O)$ with $g \le f$. Furthermore, $\hat f$ can be constructed as follows:\\
For each $n\in \N$, define $f_n$ on $O$ by
\begin{align}\label{edgar}
f_n(x) = \inf \{\E [f(F_n)] : \,&(F_i)^n_{i=0} \text{ is a spherical martingale in $O$ with } F_0 = x  \\ &\text{ and }
F_n(\P) \text{ is compactly supported in } O\}. \nonumber
\end{align}
Then, the sequence $(f_n)^{\infty}_{n=0}$  decreases pointwise to $\hat f$.  
\end{lemma}
\begin{proof} 
%{\red 
We will use an equivalent form of \eqref{edgar}:
\begin{align*}
\text{  $f_0 = f$ }
\quad  \text{and for $n \geq 1$},\quad 
f_n (x) = \text{\rm inf } \bigg\{ \int f_{n-1} (x + r y) \,dU(y) \bigg\}
\end{align*}
where the infimum is taken over all $r \geq 0$ such that  $\{ x + r \overline{B}\} \subset O$. %{\red 
Here, $\{ x + r \overline{B}\} $ denotes  the closed ball of radius $r$ around $x$. 

First note that the sequence $(f_n)^{\infty}_{n=0}$ is decreasing. To show the upper-semicontinuity of $\hat{f}$, we proceed inductively and assume $f_{n-1}$ is upper semicontinuous. If $(x_k)^{\infty}_{k=0}$ in $O$ is such that $\lim_{k \rightarrow \infty} x_k = x_0$, and $r \geq 0$ is such that  $\{ x_0 + r \overline{B}\} \subset O$,  where $\overline{B}$ is the closed unit ball, then there is $k_0$ such that $\{ x_k + r \overline{B}\} \subset O$ for $k \geq k_0$. The upper semicontinuous function $f_{n-1}$ is bounded above on the relatively compact set 
$\cup^{\infty}_{k=k_0} \{ x_k + r \overline{B}\}$ and, for every $z \in \overline{B}$,
$f_{n-1} (x_0 + r z) \geq \limsup_{k \rightarrow \infty} f_{n-1} (x_k + r z)$. Hence by Fatou's lemma, \begin{align*}
 \int f_{n-1} (x_0 + r y) \,dU(y) \geq \limsup_{k \rightarrow \infty}  \int f_{n-1} (x_k + r y) \,dU(y) 
 \geq \limsup_{k \rightarrow \infty} f_n (x_k).
 \end{align*}
Thus $f_n (x_0) \geq \displaystyle\limsup_{k \rightarrow \infty} f_n(x_k)$, hence showing that $f_n$ and consequently $\hat f$ is upper-semicontinuous.

Let now $g$ be a subharmonic function on $O$ with $g \leq f$. Again, inductively, assuming  that $g \leq f_{n-1}$, then for $\{ x_0 + r \overline{B}\} \subset O$, 
\begin{align*}
\int f_{n-1} (x_0 + r y) \,dU(y) \geq  \int g (x_0 + r y) \,dU(y)\, \geq \,g(x_0),
\end{align*}
and so $f_n (x_0) \geq g (x_0)$. Hence $\hat{f} \geq g$. Finally, for $\{ x_0 + r \overline{B}\} \subset O$,  we get from monotone convergence
\begin{align*}
\hat{f} (x_0) = \lim_{n \rightarrow \infty} f_n (x_0) \leq \lim_{n \rightarrow \infty} \int f_{n-1} (x_0 + r y) \,dU(y) =  \int \hat{f} (x_0 + r y) \,dU(y).
\end{align*}
This shows that $\hat{f}$ is subharmonic, and the proof of the lemma is complete.
\end{proof}

Let now ${\rm Lip}^*(O)$ be the space of all finite measures in $O$ with finite first moments.
\begin{lemma}\label{denseness} Let $\mu, \nu \in \P_2(O)$ be such that $\mu \prec_{s_2(O)}\nu$, and consider the following subset of $\P_2(O)$,
\begin{align*}
\Phi = \{ F_n(\P) : (F_i)^n_{i=0} \,\hbox{is a spherical martingale valued in $O$ with $F_0\sim \mu$\}.}
\end{align*}
Set $\tilde \nu:= (1+|x|)\nu$ and 
$$
\tilde \Phi = (1+|x|)\Phi:= \{\tilde \sigma \,|\, \tilde \sigma = (1+|x|)\sigma \hbox{ for some  $\sigma \in \Phi$\}.}
$$
Then, $\tilde \nu$ is in the weak$^*$-closure of $\tilde \Phi$ in ${\rm Lip}^*(O)$.
\end{lemma}
\begin{proof}
Observe first that $\Phi$ is convex. Indeed, if $(F_i')^n_{i=0}$ and $(F_i'')^m_{i=0}$ are two spherical martingales  with $F_0' \sim \mu$, $F_0'' \sim \mu$, 
we may assume $n=m$, then define a spherical martingale $(F_i)^{n+1}_{i=0}$ by letting $F_0 = F_1 \sim \mu$ 
and for $1 \leq i \leq n$,
\begin{align*}
F_{i+1} (X_1,X_2,...,X_{i+1}) = 
\begin{cases}
F_i' (X_2,...,X_{i+1}) \text{\, if \,$X_1$ is in the upper hemisphere,}\\
F_i'' (X_2,...,X_{i+1}) \text{\, if \,$X_1$ is in the lower hemisphere.}
\end{cases}
\end{align*}
Clearly $F_{n+1} (\P) = \{F_{n}' (\P) + F_{n}'' (\P)\} / 2$ and hence $\Phi$ is convex, and therefore $\tilde \Phi = (1+|x|)\Phi$ is convex in ${\rm Lip}^*(O)$.

If now the statement of the lemma were false, then by the Hahn-Banach theorem we can find a Lipschitz function $f$ on $O$ and real numbers $a < b$ such that
\begin{align*}
\int g \,d\nu \leq a, \, \text{while}\,  \int  g \circ F_n \,d\P  \geq b \,\text{ for every $F_n(\P) \in \Phi$},
\end{align*}
where $g(x) = (1+|x|)f(x)$. But since every element in $\Phi$ has initial distribution $\mu$, then by Lemma \ref{shenvelope} and \eqref{edgar}, we have $\int \hat{g} \,d\mu \geq b$ and therefore $a \geq \int g \,d\nu \geq \int \hat{g} \,d\nu \geq \int \hat{g} \,d\mu \geq b$, which is a contradiction.
\end{proof}
\begin{proof}[\bf Proof of Proposition~\ref{prop1}]
By Lemma \ref{denseness}, we have a sequence $\{\nu_n\} \subset \Phi$ such that $\int |x|^2 d\nu_n(x) \rightarrow \int |x|^2 d\nu(x)$. We know that $\nu_n = {\rm Law}(B_{\tau_n})$ for a sequence of stopping times ${\tau_n}$ and a Brownian motion $B$ with initial law $\mu$.  Hence in particular $\E \tau_n = \E |B_{\tau_n}|^2 =  \int |x|^2 d\nu_n(x) \leq V$ for some constant $V$ and for all $n$. The sequence $({\tau_n})$ is then tight, and it is standard that it has a convergent subsequence to a --possibly randomized-- stopping time $\tau$ in such a way that $\E f(B_{\tau_k}) \rightarrow \E f(B_\tau)$ for every $f$ continuous and bounded function on $\R^d$. (see for example \cite{BaCh74} or  \cite{bch}). In other words, ${\rm Law}(B_{\tau_k}) \rightarrow {\rm Law}(B_\tau)$, and therefore, $B_\tau \sim \nu$. Let $\tau_O$ be the first exit time of $(B_t)$ from $O$ and note that $\tau_k = \tau_k \wedge \tau_O$ by the definition of $\Phi$, and therefore $\tau$ also satisfies $\tau= \tau \wedge \tau_O$. 
 Notice also that $\E \tau = \E |B_{\tau}|^2 \leq V$ as well, hence the martingale $(B_{\tau \wedge t})_{t \ge 0}$ is bounded in $L^2$. 
 \end{proof}

Lastly, we prove the duality result announced in the introduction. 

\begin{proof}[\bf Proof of Theorem \ref{prop: no gap SMT}] By a standard argument, it is enough to prove the theorem for continuous cost $c$, so let us assume this. Clearly, the left-hand side is smaller than or equal to  the right-hand side since for every $\pi \in {\rm SMT}(\mu, \nu)$ (if exists) and for every $(\alpha,\beta, p) \in {\cal K}_c$, we have $\int p(x, y) d\pi (x, y) \geq 0$.  

For the reverse inequality, we first note that Kantorovich duality for the standard optimal transport problem with a continuous cost $c(x, y) - p(x, y)$ yields
\begin{align*}
  \sup_{ (\alpha, \beta, p) \in {\cal K}_c} \left(  \int \beta d\nu - \int \alpha d\mu  %\int \alpha d\mu + \int \beta d\nu 
   \right)
  = \sup_{p \in \Psi}  \inf_{\pi \in \Pi(\mu, \nu) } \int (c-p) d\pi. 
  \end{align*}
Now it is standard to apply a min-max theorem (see e.g. \cite[Theorem 2]{BeHePe11}) to interchange the order of inf and sup, and thereby  obtain
  \begin{align}\label{eq: inf sup 2}
  \sup_{ (\alpha, \beta,p) \in {\cal K}_c} \left( \int \beta d\nu - \int \alpha d\mu  \right)= \inf_{\pi \in \Pi(\mu, \nu)} \sup_{p \in \Psi} \int (c-p) d\pi.  
 \end{align}
Note that the supremum over $p$  can be finite only when $\int p(x, y) d\pi (x, y) \ge 0$, %$\int p(x, y) d\pi (x, y) = 0$, 
since otherwise we can replace $p$ with $\lambda p$ for some $\lambda >0$, %$\lambda \ne 0$, 
making the value of the integral $\int (c - p) d\pi$ arbitrarily large. Therefore, by Theorem \ref{thm:main2}, the infimum in \eqref{eq: inf sup 2} can be restricted  to $\pi \in {\rm SMT}(\mu, \nu)$, and we obtain
\begin{align*}
  \sup_{ (\alpha, \beta,p) \in {\cal K}_c} \left(  \int \beta d\nu - \int \alpha d\mu 
  \right)  = \inf_{\pi \in {\rm SMT}(\mu, \nu) } \sup_{p \in \Psi} \int (c - p )d\pi.
\end{align*}
Finally, since $0 \in \Psi$, the last expression is greater than or equal to 
$$\inf \left\{ \int c(x, y) d\pi \ | \ \pi \in {\rm SMT}(\mu, \nu) \right\}.$$
Together with Corollary \ref{same}, 
 this completes the proof of the duality for the subharmonic martingale optimal transport problem. 
\end{proof}

 \section{Appendix: The radially symmetric monotonicity principle}\label{rmonoproof}
 {
In this section we use the dual attainment result \cite{GKP-Monge} to provide a proof of  Proposition~\ref{rmonoprin}.  
\begin{proof}[\bf Proof of Proposition \ref{rmonoprin}] 
 Let $O$ be an open ball containing $\supp(\mu+\nu)$.
First,  we can assume without loss of generality that $\supp \mu \cap \supp \nu =\emptyset$. 
Indeed, due to the assumption $\mu \wedge \nu =0$ and $\mu (\partial \supp \mu) =0$, for each $\mu$-a.e. $x, x'$, we can consider the restriction of $\mu$ to $\mu^+$ on an open set containing $x, x'$ outside  $\supp \nu$ and let $\nu^+$ be the corresponding target measure under the stopping time $\tau$.  Then   $\supp \mu^+ \cap \supp \nu^+ =\emptyset$. It will be sufficient to prove  the desired conclusion  of the proposition  for those $x, x' \in \supp \mu^+$, $\supp \phi_y \subset \supp \nu^+$, and $y' \in \supp \nu^+$. 

Now, from \cite[Theorem 4.1 and the proof of Theorem 7.1 (1)]{GKP-Monge} we have the existence of the optimizers for the dual problem. Here, note that the theorems in \cite{GKP-Monge} can be used because $\supp \mu \cap \supp \nu =\emptyset$,  the function $c(x, y) =\pm |x-y|^\alpha$,  $\alpha >0$,  is $C^2$ for those $x, y \in \supp \mu \times \supp \nu$; this cost  can be modified  by adding an appropriate subharmonic function  $h(y)$, with large $\Delta h$, to an equivalent  cost   $y \mapsto c(x, y) + h(y)$ 
with $0\le  \Delta_y [ c(x, y) + h(y)]$ which is $C^2$ on $\supp \mu \times \supp \nu$. Then, as in the  proof of \cite[Theorem 7.1 (1)]{GKP-Monge} this can be extended to a $C^2$ subharnonic function on $O\times O$ without changing the optimal solution.  The result \cite[Theorem 4.1]{GKP-Monge} can be applied to this cost to get the dual optimizers, first $\beta \in H_0^1(O)\cap LSC(O)$, then   
\begin{align}\label{eqn:J}
J(x, y) :=\sup_{\xi \le \tau_O} [  \beta (B^y_\xi) - c(x, B^y_\xi) ],
\end{align}
where $\xi$ is a stopping time for Brownian motion, and $\tau_O$ is the exit time from the domain $O$. Set $\alpha(x) = J(x,x)$ and $p(x,y) = J(x,x) - J(x, y)$.  
 
From the radial symmetry of $\mu$ and $\nu$, and the expression of the dual problem, we can assume that the dual optimizer $\beta$ is radially symmetric. To see this, recall that $\mathfrak{M}$ is the group of $d\times d$ orthogonal matrices and $\mathcal{H}$ is the Haar measure on $\mathfrak{M}$. Then, for each $M \in \mathfrak{M}$, 
\begin{align*}
 J(Mx, My) =  \alpha (Mx) - p(Mx, My) =\sup_{M\xi \le \tau_O} \left[  \beta (B^{My}_{M\xi}) - c(Mx, B^{My}_{M\xi}) \right].
\end{align*}
Since both Brownian motion and the cost $c(x, y) = \pm|x-y|^\alpha$ are invariant under the orthogoal group (isometries of $\R^d$) we see that for 
\begin{align*}
 \bar J (x, y) = \int_{M\in \mathfrak{M}} J (Mx, My) d\mathcal{H}(M) \quad \hbox{and}  \quad  \bar \beta(y) = \int_{M\in \mathfrak{M}} \beta (My) d\mathcal{H}(M), 
\end{align*}
it holds that
\begin{align*}
 \bar J (x, y) = \sup_{\xi \le \tau_O} \left[ \bar\beta (B^y_\xi) - c(x, B^y_\xi)\right]. 
\end{align*}
Notice that $\bar \beta$, $\bar \alpha(x) := \bar J(x,x)$ and  $\bar p(x, y) := \bar J(x,x) - \bar J(x,y)$ are admissible and they have the same optimal dual value as $\int \beta d\nu - \int \alpha d\mu$, thus they are dual optimizers. 

The rest of the proof is similar to that of  \cite[Theorem B.1]{GKP-Monge}. We will use the radial symmetry of $\beta$. We notice that the desired conclusion of our theorem is equivalent to the following:

\begin{claim} Let $\pi$ be the measure with $(B_0, B_\tau) \sim \pi$. Consider a randomized stopping time $\xi$ with $B_0 \sim \mu$, $(B_0, B_\xi) \sim \sigma$ where the disintegration $\sigma(dx, dy) = \sigma_x (dy) \mu (dx)$ satisfies 
\begin{equation} \hbox{$0 \prec_{s(O)} \sigma_x \prec_{s(O)}  \pi_x$ 
for  each $\mu$-a.e. $x$,}
\end{equation}
which means that $\sigma \le \tau$, and $\tau-\sigma$ is a (randomized) stopping time.

Then, for $\pi$-a.e.  $(x', y')$ and $\sigma$-a.e. $(x, y)$, it holds that  if   $|y|=|y'|$,  then the stopping time $\tau-\xi$ restricted to paths with $B_\xi=y$ satisfies
  \begin{align*}
  c(x', y' )  + \mathbb{E}\big[c(x, B^{y}_{\tau-\xi} )\big] \le 
   \mathbb{E}\big[c(x', B^{y'}_{\tau'}) \big] + c(x, y), 
  \end{align*}
 for any stopping time $\tau'$ such that $B^{y'}_{\tau'} \sim \phi_{y'} \cong_R \psi_y \sim B^y_{\tau -\xi}  .$ 
\end{claim}

To prove this claim notice that from \cite[Theorem 4.7]{GKP-Monge}, we have for $\pi^*$-a.e. $(x', y')$ and $\sigma$-a.e. $(x, y)$, 
\begin{align}  
 J(x', y') & = \beta (y') - c(x', y'), \label{yes1}\\    
 J(x, B^{y}_{\tau-\xi}) = J(x, B^{x}_{\tau})& = \beta (B^{x}_{\tau}) -  c(x, B^{x}_{\tau})=\beta (B^{y}_{\tau-\xi}) -  c(x, B^{y}_{\tau-\xi}). \label{yes2}
\end{align}
On the other hand, 
from the definition  \eqref{eqn:J}, we have  
\begin{align}\label{OK1}
J(x', y')   \ge  \mathbb{E}\Big[  \beta (B^{y'}_{\tau'}) - c(x', B^{y'}_{\tau'} ) \Big] =  \mathbb{E}\Big[  \beta (B^{y}_{\tau-\xi}) - c(x', B^{y'}_{\tau'}) \Big] 
\end{align}
 where the first inequality holds for any $\tau'$ and the second one holds  in particular for those with $B^{y'}_{\tau'} \sim \phi_{y'} \equiv_R \psi_y \sim B^{y}_{\tau-\xi}$ due to radial symmetry of $\beta$.
Notice that from \cite[Theorem 4.7 (4.4)]{GKP-Monge} we have
\begin{align}\label{OK2}
\mathbb{E}\Big[ J (x, B^{y}_{\tau-\xi} ) \Big]  = J(x, y).
\end{align}
Moreover, 
\begin{align*}
J(x, y) &\ge \beta (y) - c(x, y)
= \beta(y') - c(x, y) 
\end{align*}
where the last equality uses the radial symmetry of $\beta$ and the fact that $|y|=|y'|$.

Taking the expectation in \eqref{yes2}, we see that the left hand sides of (\ref{yes1}), (\ref{yes2})
are equal to the left hand sides of inequalities (\ref{OK1}) and (\ref{OK2}). Now, subtract the sum of the two equations from the sum of the two inequalities and cancel the terms  $\beta(y)$ and $\mathbb{E}\big[  \beta(B^{y}_{\tau-\xi})\big]$, to obtain % the inequality
\begin{align*}
  0 \ge  - c(x, y) -  \mathbb{E}\Big[  c(x', B^{y'}_{\tau'})\Big]
  + c(x', y') + \mathbb{E}\Big[  c(x, B^{y}_{\tau-\xi})  \Big], 
\end{align*}
hence completing the proof. 
 \end{proof}

}

\bibliographystyle{plain}
%\phantomsection % so that hyper ref jumps to the right place
%\cleardoublepage % To get correct page numbering in the ToC
%\addcontentsline{toc}{chapter}{Bibliography} % To add this entry in the ToC{ 
%\bibliography{biblio} %make the bibliography if using BibTex

% AOS,AOAS: If there are supplements please fill:
%\begin{supplement}[id=suppA]
%  \sname{Supplement A}
%  \stitle{Title}
%  \slink[doi]{10.1214/00-AOASXXXXSUPP}
%  \sdatatype{.pdf}" 
%  \sdescription{Some text}
%\end{supplement}

\end{document}